\documentclass[11pt]{amsart}
\usepackage{amssymb}
\usepackage{lscape}
\usepackage{tikz}
\usepackage{t1enc}
\usepackage[magyar, english]{babel}
\newtheorem{thm}{Theorem}[section]

\newtheorem{lemma}[thm]{Lemma}
\newtheorem{cor}[thm]{Corollary}
\newtheorem{rem}[thm]{Remark}
\newtheorem{conj}[thm]{Conjecture}

\newcommand{\Z}{{\mathbb Z}}
\newcommand{\Q}{{\mathbb Q}}

\newcommand{\R}{{\mathbb R}}

\numberwithin{equation}{section}

\begin{document}
	
\title{Rotation on the digital plane}
	
\author{Carolin Hannusch}
\address{Faculty of Informatics, University of Debrecen, Hungary}
	\email{hannusch.carolin@inf.unideb.hu}
	
\author{Attila Pethő}
\address{Faculty of Informatics, University of Debrecen, Hungary}
\email{petho.attila@unideb.hu}

\begin{abstract}
Let $A_{\varphi}$ denote the matrix  of rotation with angle $\varphi$ of the Euclidean plane, FLOOR the function, which rounds a real point to the nearest lattice point down on the left and ROUND the function for rounding off a vector to the nearest node of the lattice. We prove under the natural assumption $\varphi\not= k\frac{\pi}{2}$ that the functions $FLOOR \circ A_{\varphi}$ and $ROUND \circ A_{\varphi}$ are neither surjective nor injective. More precisely we prove lower and upper estimates for the size of the sets of lattice points, which are the image of two lattice points as well as of lattice points, which have no preimages. It turns out that the density of that sets are positive except when $\sin \varphi \not= \pm \cos \varphi + r, r\in \Q$.
\end{abstract}

\maketitle

\section{Introduction}
The digital plane is a lattice whose elements are points with integer coordinates, the so called
lattice points. The values of continuous functions can be represented only approximatively, rounding its computed value to a lattice point. Rounding is a mapping $q\; : \; \R^2 \mapsto \Z^2$. The discrete  variant of the function $f\;:\; \R^2 \mapsto \R^2$ is $q\circ f$. Of course there are plenty discrete variants of $f$. One of the most studied function of the plane is the rotation, which is a $2\times 2$ real matrix with eigenvalues $\cos \varphi + i \sin \varphi$ and $\cos \varphi - i \sin \varphi$. Let us denote it by $Q_{\varphi}$. There exists an invertible real matrix $Q$ such that $Q_{\varphi} = Q A_{\varphi} Q^{-1}$ with
$$
A_{\varphi}= \left(
\begin{array}{rr}
\cos \varphi & -\sin \varphi\\
\sin \varphi & \cos \varphi
\end{array}
\right).
$$
For $(x,y)\in \R^2$ the sequence of points $\{f^n(x,y)\}_{n=0}^{\infty}$ is called an {\it orbit}\footnote{Some author calls this sequence {\it trajectory}.} of $f$ generated by $(x,y)$. The orbits of $Q_{\varphi}$ generated by any non-zero points lie on an ellipse, those of $A_{\varphi}$ on the unit circle. Plainly $Q_{\varphi}$ is a bijective mapping on $\R^2$, but usually not bijective on $\Z^2$. Combining it with a rounding, which results the function $q \circ Q_{\varphi}$ we obtain a mapping of $\Z^2 \mapsto \Z^2$, which is a {\it discrete rotation} in wide sense. Many interesting and hard questions appear: is a discrete rotation injective, surjective or bijective? How are its orbits? All orbits of $Q_{\varphi}$ are bounded, but this is not at all clear for its discrete variants. The investigation of such questions have long tradition, see the early example \cite{DeVogelaere}.

FLOOR\footnote{Rounds a real point to the nearest lattice point down on the left.}, ROUND\footnote{Stands for rounding off a vector to the nearest node of the lattice. } and TRUNC \footnote{Denotes the coordinate-wise truncation of the fractional part of a vector towards the zero point.} are eminent examples of rounding functions. Discretizing the rotation with them the resulted mappings $\Z^2 \mapsto \Z^2$ are more or less  different. Kozyakin et al \cite{Kozyakin-Kuznetsov-Pokrovskii-Vladimirov} gave a good overview on the results concerning discretized rotations, especially on $ROUND \circ A_{\varphi}$ and $TRUNC \circ A_{\varphi}$. Diamond et al \cite{Diamond_etal} proved that if $\varphi\not= k\frac{\pi}{2}, k\in \Z$ then all orbits of $TRUNC \circ A_{\varphi}$ eventually gets into the zero point. The situation is very different with $ROUND \circ A_{\varphi}$. Kozyakin et al \cite{Kozyakin-Kuznetsov-Pokrovskii-Vladimirov} proved among others that if the rotation angle $\varphi$ is such that the rows of all the nonnegative powers of the matrix $A_{\varphi}$ are rationally independent then the density of lattice points with empty full preimages is positive. They used measure theoretic approach, which allow to prove much more general results too. For other probabilistic results on discrete rotations we refer to \cite{DeVogelaere,Vivaldi:94,Vivaldi:06,Kouptsov-Lowenstein-Vivaldi:02}.

In former investigations of the second author with different coauthors  \cite{Akiyama-Borbely-Brunotte-Pethoe-Thuswaldner:04,Akiyama-Brunotte-Pethoe-Steiner:06,
Akiyama-Brunotte-Pethoe-Steiner:07,Akiyama-Pethoe:13} a kind of discrete rotation appears as a natural generalization of positional number systems. It was $FLOOR\circ B_{\varphi}$ with
$$
B_{\varphi}= \left(
\begin{array}{rr}
0 & 1\\
-1 & -\lambda
\end{array}
\right), \quad \lambda = -2\cos \varphi.
$$

We come back to this function later, but before we discuss some properties of the FLOOR function. It commutes with the additive group of translations of $\Z^2$ and the full preimage of zero is $[0,1[\times [0,1[$, which is Jordan measurable, thus FLOOR, like ROUND, is a {\it quantizer} in the sense of \cite{Kozyakin-Kuznetsov-Pokrovskii-Vladimirov}. Hence the discretized rotation FLOOR$\circ A_{\varphi}$ has similar properties as ROUND$\circ A_{\varphi}$. This holds, among others, for the above mentioned property of preimages.

\medskip

In this note we prove under the natural assumption $\varphi\not= k\frac{\pi}{2}$ that the function $FLOOR \circ A_{\varphi}$ is neither surjective nor injective. More precisely we prove lower density estimates for the sets of lattice points, which are the image of two lattice points as well as of lattice points, which have no preimages. It turns out, see Theorems \ref{t:inj} and \ref{t:surj}, that these densities are positive except when $\sin \varphi \not= \pm \cos \varphi + r, r\in \Q$. This means that the number of such lattice points lying in a box symmetric to the origin and of side length $2M+1$ is $O(M^2)$. In contrast in the exceptional case this number is only $O(M)$. In Section \ref{s:round} we indicate that the same results hold to $ROUND \circ A_{\varphi}$ too. We use in the proof elementary results of uniform distribution theory and properties of primitive Pythagorean triplets. Our results are more precise than those in \cite{Kozyakin-Kuznetsov-Pokrovskii-Vladimirov}.

There are discrete rotations which are bijective. Trivial examples are $FLOOR \circ A_{\varphi}$ with $\varphi=k\frac{\pi}{2}, k\in \Z$. More interesting are the functions $FLOOR\circ B_{\varphi}, 0\le \varphi<2\pi$. Reeve-Black and Vivaldi \cite{ReeveBlack-Vivaldi} claim that a generic discrete rotation is neither injective nor surjective. Our results justify this claim for the function $FLOOR \circ A_{\varphi}$. To prove similar characterization for $FLOOR \circ Q_{\varphi}$ is a challenging problem. We expect that apart the previous examples only the transpose of $B_{\varphi}$ lie in the exceptional set.

\medskip
Despite many efforts and interesting results, we have deterministic knowledge only on the orbits of $FLOOR\circ B_{\varphi}$. For the eleven values $2\cos \varphi = \lambda = 0,\pm1, (\pm1 \pm \sqrt{5})/2, \pm\sqrt{2}, \pm\sqrt{3}$ all orbits are periodic, see \cite{Lowenstein_etal,Akiyama-Brunotte-Pethoe-Steiner:06, Akiyama-Brunotte-Pethoe-Steiner:07}. Generally it was proved by Akiyama and Peth\H{o} \cite{Akiyama-Pethoe:13} that for any $\varphi$ there are infinitely many periodic orbits. This is still far from the conjecture that all orbits are periodic see \cite{Akiyama-Brunotte-Pethoe-Steiner:06}.

The matrix $A_{\varphi}$  through $A_{\varphi}\cdot(a,b), (a,b)\in \R^2$
\footnote{To be precise we had to write $A_{\varphi}\cdot(a,b)^T$ instead of $A_{\varphi}\cdot(a,b)$, where $(a,b)^T$ denote the transpose of the vector $(a,b)$, i.e., a column vector. As in the article we should do this often, and from the context it will be clear whether the actual vector is a row or a column vector, we avoid this extra notation.}
induces a linear mapping on $\R^2$, which we will denote by $A_{\varphi}$ too.

FLOOR and integer part $\lfloor.\rfloor$ are the same functions, in the sequel we will use the later. To simplify our notation we define $r_{\varphi}\; : \; \Z^2 \mapsto \Z^2$ by
\begin{eqnarray*}
  r_{\varphi}(a,b) &=& \lfloor  A_{\varphi}(a,b) \rfloor\\
   &=& (\lfloor a \cos \varphi - b \sin \varphi\rfloor, \lfloor a \sin \varphi + b \cos \varphi\rfloor).
\end{eqnarray*}

We computed the orbits of $r_{\varphi}$ for many choice of the angle and the starting point and found always periodicity. For the angle $\varphi = \frac{\pi}{4}$ we found infinitely many starting points which generate short periodic orbit, see Theorem \ref{45fok}. Based on our numerical and theoretical results we propose the following conjecture
\begin{conj}
	Every orbit of $r_{\varphi}$ is periodic.
\end{conj}

 We also use the fractional part function, i.e, $\{x\}= x-\lfloor x \rfloor$. Both functions will be applied
 coordinate wise to the points of the real vector spaces. We use the same notation to these extended functions. Let
 $U=[0,1[\times[0,1[$ and $\bar{U}=[0,1]\times[0,1]$, then obviously $r_{\varphi}(a,b)=(x,y)\in \Z^2$ if and only
 if $A_{\varphi}\cdot(a,b)=(x,y)+u$ for some $u\in U$. The third equivalent expression is
 $\{A_{\varphi}\cdot(a,b)\} = A_{\varphi}\cdot(a,b)- r_{\varphi}(a,b)$.

\section{Preliminary results}

In order to prove our main results we need some tools from uniform distribution theory. Let ${\bf a}=(a_1,\ldots,a_n), {\bf
b}=(b_1,\ldots,b_n)\in \R^n$ be such that $0\le a_j<b_j\le 1,\; j=1,\ldots,n$ then we set $B_{{\bf
a,b}}=[a_1,b_1[\times \ldots \times[a_n,b_n[$. This is a box with side lengths $b_1-a_1,\ldots,b_n-a_n$, whose
volume is plainly $\prod_{j=1}^{n}(b_j-a_j)$. For a sequence of $n$-dimensional real vectors $X=({\bf x}_m)$ set
$$
A(X, B_{{\bf a,b}},N) = |\{m\;:\; 0\le m\le N, \{{\bf x}_m\}\in B_{{\bf a,b}}\}|,
$$
where $|S|$ denotes the cardinality of the set $S$.

The sequence $X$ is called {\it uniformly distributed modulo $1$}, shortly {\it uniformly distributed} if
$$
\lim_{N\to \infty} \frac{A(X,B_{{\bf a,b}},N)}{N} = \prod_{j=1}^{n}(b_j-a_j)
$$
holds for all ${\bf a,b}\in \R^n$ with the above property. Notice that if $X$ is uniformly distributed then there
exist for any $B_{{\bf a,b}}$ a constant $c= c_{\bf a,b}>0$ such that
\begin{equation}\label{eq:ud}
  A(X, B_{{\bf a,b}},N)> c N
\end{equation}
holds for all large enough $N$. Indeed, one may set $c = \prod_{j=1}^{n}(b_j-a_j)-\varepsilon$ for some
$\varepsilon>0$.

The following theorem is a bit modified version of Theorem I, p.64. of Cassels \cite{Cassels} and it plays a crucial role in this paper.

\begin{thm}\label{th:Kronecker}
  Let $L_j({\bf x})$ for $1\le j\le m$ be homogeneous linear forms in the $n$ variables ${\bf x}=(x_1,\ldots,x_n)$.
  Suppose that the only set of integers $u_1,\ldots,u_m$ such that
  $$
  u_1L_1({\bf x})+\dots + u_mL_m({\bf x})
  $$
  has integer coefficients in $x_1,\ldots,x_n$ is $u_1=\dots=u_m=0$. Then the set of vectors ${\bf z}^{({\bf x})} =
  (L_1({\bf x}),\dots,L_m({\bf x}))$ for integral ${\bf x}$ is uniformly distributed modulo $1$.
\end{thm}

Now we formulate the main lemma of this paper.

\begin{lemma} \label{l:main}
  Let $0<t_1,t_2\le 1$ and set $L_1(x_1,x_2) = x_1 \cos \varphi - x_2 \sin \varphi$, $L_2(x_1,x_2) = x_1 \sin
  \varphi + x_2\cos \varphi$. If $\varphi \not= k\frac{\pi}{2}, k \in \Z$ then there exist constants
  $c_1,c_2>0$ depending only on $\varphi,t_1,t_2$ such that the number of solutions $(x_1,x_2)\in \Z^2, |x_1|,|x_2|\le M$ of the system of inequalities
  \begin{eqnarray}
    0\le \{L_1(x_1,x_2)\}  &<& t_1 \label{e:e3}\\
    0\le \{L_2(x_1,x_2)\} &<&t_2 \label{e:e4}
  \end{eqnarray}
 lie between $c_1M^2$ and $c_2M^2$ except when $\cos\varphi = \pm \sin\varphi +r, \; r\in \Q$ in which case it lies between $c_1M$ and $c_2M$.

  The same statement holds for the number of solutions in pairs of odd integers.
\end{lemma}

\begin{proof}
There are only $(2M+1)^2$ integer pairs with $|x_1|,|x_2|\le M$, thus the upper estimate $c_2M^2$ is obvious. Hence in the sequel we have to deal with the upper bound in the exceptional case and with the lower bound.

In the proof we have to distinguish three cases according the arithmetic nature of $\cos \varphi$ and
$\sin\varphi$.

\medskip

{\bf Case 1.} $1,\sin\varphi,\cos\varphi$ are $\Q$-linearly independent. Then the linear forms
$L_1(x_1,x_2),L_2(x_1,x_2)$ satisfy the assumptions of Theorem \ref{th:Kronecker}, thus the points
$(\{L_1(a,b)\},\{L_2(a,b)\} ) $ for $a,b\in \Z$ is uniformly distributed in $\left[0,1\right[^2.$

There are $(2M+1)^2$ points $(a,b)\in \Z^2$ with $|a|,|b|\le M$, thus setting $B=[0,t_1[\times [0,t_2[$ and
$N=(2M+1)^2$ in \eqref{eq:ud} we obtain the statement immediately.

If $(\{L_1(a,b)\},\{L_2(a,b)\} ) $ for $a,b\in \Z$ is uniformly distributed in $\left[0,1\right[^2$ then the same holds if $(a,b)$ runs through a sublattice of $\Z^2$, which proves the second assertion. In Case 2 we use the uniformly distributed property of some sequence, hence for them the second assertion holds by the above remark.
\medskip

{\bf Case 2.} $\cos\varphi = r_1\sin\varphi +r_2,$ where $r_1,r_2\in \Q,$ $r_1\neq 0$, and $\sin\varphi\notin\Q.$
There exist integers $p_1,p_2,q$ with $q> 0,$ such that $r_1=\frac{p_1}{q}$ and $r_2=\frac{p_2}{q}.$ For any
$(a,b)\in \Z^2$ there exist $u,v,s,t\in \Z, \; 0\leq s,t<q$ such that $a=uq+s$ and $b=vq +t$. With these notations
we obtain
\begin{eqnarray*}
	L_1(a,b) &=& \left(a\frac{p_1}{q}-b\right) \sin\varphi + \frac{ap_2}{q} \\
	&=&(up_1-vq)\sin\varphi + \left(\frac{sp_1}{q}-t\right)\sin\varphi + \frac{ap_2}{q}
\end{eqnarray*}
and similarly
$$
L_2(a,b) = (uq+vp_1)\sin\varphi + \left(s+\frac{tp_1}{q}\right)\sin\varphi + b\frac{p_2}{q}.
$$
Fix $s,t$ and set $f_1=\left\{\left(\frac{sp_1}{q}-t\right)\sin\varphi + \frac{ap_2}{q}\right\}$ and
$f_2=\left\{\left(s+\frac{tp_1}{q}\right)\sin\varphi + b\frac{p_2}{q}\right\}.$
We have $0\leq f_1,f_2<1.$ With these notations we have to count the number of solutions of the system of
inequalities
\begin{eqnarray*}
0&\leq& \left\{(up_1-vq)\sin\varphi +f_1\right\} <t_1\\
0&\leq& \left\{(uq+vp_1)\sin\varphi +f_2\right\} <t_2
\end{eqnarray*}
in the integers $u,v$ with $|u|,|v|\le \frac{M}{q}-1$. We have to distinguish two subcases

\medskip

{\bf Case 2.1.} $p_1=\pm q$. We have $up_1-vq = p_1(u\mp v)=  \pm q(u\mp v) = \pm (uq + vp_1)$, thus the
coefficients of $\sin\varphi$ in the last inequalities are up to sign equal and may assume at most $O(M)$ different
values. As $\sin\varphi\notin \Q$ the sequence $(\{n \sin\varphi\})$ is by Theorem \ref{th:Kronecker} uniformly
distributed in $[0,1[$, thus each inequality, hence the system too, has at most $O(M)$ solutions.

Choosing $s=t=0$ the integers $a$ and $b$ are divisible by $q$, hence $f_1=f_2=0$. If, for example, $p_1=q$ then we
may assume without loss of generality $0<t_1\le t_2$. Our first inequality has at least $(t_1/2)(M/q)$ solutions in
$u,v\in \Z,\; |u|,|v|\le \frac{M}{q}-1< \frac{M}{q}$ provided $M$ is large enough.

\medskip

{\bf Case 2.2.} $p_1\not=\pm q$. Setting $D=q^2-p_1^2$ this is equivalent to $ D\not= 0$. As $\sin\varphi\notin \Q$
and $D\in \Z$ we have $D \sin\varphi\notin \Q$, hence the sequence $(\{nD\sin\varphi\})$ is uniformly distributed in
$[0,1[$. Thus the inequality
$$
0\le u D\sin\varphi + f \le t
$$
has for any fixed $0\le f<1,\; 0<t<1$ at least $O(M')$ solutions in $u\in \Z,\; |u|\le M'$, where the positive
constant indicated by the $O$ notation depends only on $\sin\varphi$ and on $t$. Hence the system of the (independent)
inequalities
\begin{eqnarray*}
0&\leq& \left\{u_1 D\sin\varphi +f_1\right\} <t_1\\
0&\leq& \left\{u_2 D\sin\varphi +f_2\right\} <t_2
\end{eqnarray*}
has at least $O(M'^2)$ solutions in $(u_1,u_2)\in \Z^2,\; |u_1|,|u_2|\le M'$.

For all such pairs the system of linear equations
\begin{eqnarray*}
up_1-vq&=&u_1(q^2-p_1^2)\\
uq+vp_1&=&u_2(q^2-p_1^2)
\end{eqnarray*}
has unique solution in $u,v\in \Z$, i.e. $(u,v)$ solves our original system of equations. Using Cramer's rule we
obtain
$$
u = \frac{(u_1p_1+u_2q)(q^2-p_1^2)}{p_1^2+q^2} \quad \mbox{and} \quad v= \frac{(u_2p_1-u_1q)(q^2-p_1^2)}{p_1^2+q^2}.
$$
Thus $|u|,|v| \le 2qM'$. Hence choosing $M'=\frac{M}{2q^2}$ we can produce at least $O(M^2)$ integer solutions of
\eqref{e:e3} and \eqref{e:e4}.

\medskip

{\bf Case 3.} $\sin\varphi,\cos\varphi\in \Q.$ Although the statement is the same as in the other cases, we have to use different tools in the proof, because Theorem \ref{th:Kronecker} does not hold. We have a kind of discrete uniform distribution treated systematically in Narkiewicz \cite{Narkiewicz}. First we set $\sin\varphi=\frac{p_1}{q},\; \cos\varphi=\frac{p_2}{q},$ where
$p_1,p_2,q\in \Z,\; p_1,p_2,q\neq 0.$ Then $$1=\sin^2\varphi+\cos^2\varphi = \frac{p_1^2}{q^2}+\frac{p_2^2}{q^2},$$
which implies $p_1^2+p_2^2=q^2,$ i.e.~$p_1,p_2,q$ is a Pithagorean triple, and we may assume that it is primitive,
i.e.~$\gcd(p_1,q)=\gcd(p_2,q)=1.$ Then there are $u,v\in \Z,$ $u\not\equiv v \pmod 2,$ $\gcd(u,v)=1$ such that
$p_1=u^2-v^2, p_2=2uv$ or $p_1=2uv, p_2=u^2-v^2$ and $q=u^2+v^2.$ We work out in the sequel only the first
possibility, because the alternative case can be handled analogously.

With these notations and $(a,b)\in \Z^2$ we have
\begin{eqnarray*}
L_1(a,b) &=& a\frac{2uv}{u^2+v^2} - b\frac{u^2-v^2}{u^2+v^2}\\
L_1(a,b) &=& a\frac{u^2-v^2}{u^2+v^2} + b\frac{2uv}{u^2+v^2}.
\end{eqnarray*}

Because $\gcd(2uv,u^2+v^2)=1$ there exists an integer $0<h<u^2+v^2=q$ such that
\begin{equation}\label{e:5}
	2uvh\equiv u^2-v^2 \pmod{u^2+v^2}.
\end{equation}

For fixed $a,b\in\Z$ there exist $d,d_1\in\Z$ such that
$a-hb=d(u^2+v^2)+d_1,$ where $0\leq d_1<u^2+v^2.$ Then
$$a(u^2-v^2)+2buv = bh(u^2-v^2)+d(u^4-v^4)+d_1(u^2-v^2)+2buv.$$
Multiplying this by $2uv$ and taking \eqref{e:5} into account we obtain
\begin{eqnarray*}
2uv(a(u^2-v^2)+2buv)&\equiv& 2uvhb(u^2-v^2)+2uvd_1(u^2-v^2)+b(2uv)^2 \\
&\equiv& b((u^2-v^2)^2+4u^2v^2)+2uvd_1(u^2-v^2)\\
&\equiv& b(u^2+v^2)^2+2uvd_1(u^2-v^2)\\
&\equiv& 2uvd_1(u^2-v^2) \pmod{u^2+v^2}.
\end{eqnarray*}
Since $\gcd(2uv,u^2+v^2)=1$,  we obtain
\begin{equation}\label{6}
a(u^2-v^2)+2buv \equiv d_1(u^2-v^2) \pmod{u^2+v^2}.
\end{equation}
On the other hand by \eqref{e:5} we get
\begin{eqnarray*}
2uva-(u^2-v^2)b &\equiv& b(2uvh-(u^2-v^2))+2uvd_1\\
&\equiv& 2uvd_1 \pmod{u^2+v^2}.
\end{eqnarray*}
Hence
$$
\{L_1(a,b)\} = \left\{\frac{2uvd_1}{u^2+v^2}\right\}
$$
and
$$
\{L_2(a,b)\} = \left\{\frac{(u^2-v^2)d_1}{u^2+v^2}\right\}.
$$
Choosing $d_1=0$, which implies $a= hb+d(u^2-v^2)$ we obtain that the pair $(a,b)\in \Z^2$ solves
the system of inequalities \eqref{e:e3}, \eqref{e:e4}.

Finally choosing $b,d\in \Z$ such that $|b|\le \frac{M}{2q}< \frac{M}{2h}$ and $|d|\le \frac{M}{2|p_1|} =
\frac{M}{2|u^2-v^2|}$ we obtain $|a|\le M$, hence such $(a,b)\in \Z^2$ pairs not only solves
the system of inequalities \eqref{e:e3},\eqref{e:e4}, but satisfies $|a|,|b|\le M$ too. Plainly the number of such
pairs is at least $O(M^2)$ and the lemma is completely proved.

Choosing $b$ odd $a$ is odd too for all odd or even $d$ depending on the parities of $h$ and $u^2-v^2$. The argument of the last paragraph proves the second assertion in this case.

\end{proof}

\section{Injectivity of digital rotation} \label{s:inj}

Before stating the main result of this section we introduce a notation: $T_{\varphi}(M)$ denotes the number of $(x,y)\in \Z^2$ such that $|x|,|y|\le M$ and $(x,y)$ is the image by $r_{\varphi}$ of two different grid points. If $\varphi = k\frac{\pi}{2}, k\in \Z$ then $r_{\varphi} = A_{\varphi}$, thus it is bijective, thus $T_{\varphi}(M)=0$.
Otherwise, we prove that $T_{\varphi}(M)$ tends to infinity.

\begin{thm} \label{t:inj}
If $\varphi \not= k\frac{\pi}{2}, k \in \Z$ then there exist constants $c_3,c_4>0$ depending only on $\varphi$ such that
$c_3M^2 \le T_{\varphi}(M) \le c_4M^2$ except when $\cos\varphi = \pm \sin\varphi + r,\; r\in \Q$, in which case $c_3M \le T_{\varphi}(M) \le c_4M$ hold.
\end{thm}

\begin{proof}
Let $(a_1,b_1), (a_2,b_2) \in \Z^2$ such that $r_{\varphi}(a_1,b_1)=r_{\varphi}(a_2,b_2)$. Then
\begin{eqnarray*}
  \lfloor a_1 \cos \varphi - b_1 \sin \varphi\rfloor &=& \lfloor a_2 \cos \varphi - b_2 \sin \varphi  \rfloor\\
  \lfloor a_1 \sin \varphi + b_1 \cos \varphi\rfloor &=& \lfloor a_2 \sin \varphi + b_2 \cos \varphi\rfloor,
\end{eqnarray*}
which implies
\begin{eqnarray*}
  (a_1-a_2) \cos \varphi - (b_1-b_2) \sin \varphi &=& u_1   \\
  (a_1-a_2) \sin \varphi + (b_1-b_2) \cos \varphi &=& u_2
\end{eqnarray*}
for some $-1<u_1,u_2<1$. After squaring and adding these equations we obtain
$$
(a_1-a_2)^2 + (b_1-b_2)^2 = u_1^2 + u_2^2 <2.
$$
As $a_1-a_2$ and $b_1-b_2$ are integers the right hand side can be either zero ore one. The first case is excluded because either $a_1-a_2$ or $b_1-b_2$ is nonzero.
Thus we have the following four possibilities:
$$\begin{array}{c|c|c|c|c}
	a_1-a_2 & -1 & 0 & 0 & 1\\
	\hline
	b_1-b_2 & 0 & -1 & 1 & 0
	\end{array}$$

This implies that if two grid points have the same image by $r_{\varphi},$ then they are neighbors.

Now, we show that for each $\varphi$ there exist infinitely many neighbors such that
$r_{\varphi}(a_1,b_1)=r_{\varphi}(a_2,b_2).$

First, we assume $a_1=a_2$ and $b_2=b_1+1.$ Inserting them into the starting equations we obtain
\begin{equation}\label{e1}
\lfloor a_1\cos\varphi-b_1\sin\varphi\rfloor = \lfloor a_1\cos\varphi-b_1\sin\varphi -\sin\varphi\rfloor
\end{equation}
and
\begin{equation}\label{e2}
\lfloor a_1\sin\varphi+b_1\cos\varphi\rfloor = \lfloor a_1\sin\varphi+b_1\cos\varphi +\cos\varphi\rfloor
\end{equation}

If $\varphi\in \left]\frac{3\pi}{2},2\pi\right[,$ then $\sin\varphi<0$ and $\cos\varphi>0$ and the system of
equations \eqref{e1} and \eqref{e2} holds if and only if

\begin{equation}\label{e3}
	0 \leq \{L_1(a_1,b_1)\} = \{a_1\cos\varphi-b_1\sin\varphi\} < 1 + \sin \varphi
\end{equation}
and
\begin{equation}\label{e4}
	0 \leq \{L_2(a_1,b_1)\} = \{a_1\sin\varphi+b_1\cos\varphi\} < 1 - \cos \varphi.
\end{equation}

Observe that $L_1,L_2$ are the linear forms introduced in Lemma \ref{l:main}, further setting $t_1=1 + \sin \varphi,\; t_2 = 1 - \cos \varphi$ the lemma implies our statement in this case.

The other cases can be handled similarly we give only the important data for repeating the argument.

If $a_2=a_1+1$ and $b_2=b_1,$ then \eqref{e1}, \eqref{e2} reads $[L_1(a_1,b_1)] = [L_1(a_1,b_1) + \cos \varphi]$ and $[L_2(a_1,b_1)] = [L_2(a_1,b_1) + \sin \varphi]$. If $\varphi \in ]0,\frac{\pi}{2}[$ then $\sin \varphi,\cos \varphi>0$ and setting $t_1=1 - \cos \varphi,\; t_2 = 1 - \sin \varphi$ we can apply Lemma \ref{l:main}.

If $a_2=a_1$ and $b_2=b_1-1,$ then \eqref{e1}, \eqref{e2} reads $[L_1(a_1,b_1)] = [L_1(a_1,b_1) + \sin \varphi]$ and $[L_2(a_1,b_1)] = [L_2(a_1,b_1) - \cos \varphi]$. If $\varphi \in ]\frac{\pi}{2},\pi[$ then $\sin \varphi>0$ and $\cos \varphi<0$ and setting $t_1=1 - \sin \varphi,\; t_2 = 1 + \cos \varphi$ we are done by Lemma \ref{l:main}.

Finally if $a_2=a_1-1$ and $b_2=b_1,$ then \eqref{e1}, \eqref{e2} reads $[L_1(a_1,b_1)] = [L_1(a_1,b_1) - \cos \varphi]$ and $[L_2(a_1,b_1)] = [L_2(a_1,b_1) - \sin \varphi]$. If $\varphi \in ]\pi,3\frac{\pi}{2}[$ then $\sin \varphi,\cos \varphi<0$ and setting $t_1=1 + \cos \varphi,\; t_2 = 1 + \sin \varphi$ finishes the proof by Lemma \ref{l:main}.
\end{proof}

\section{Surjectivity of digital rotation}

In the last section we proved that the digital rotation is usually not injective. Now we prove that usually it is not surjective either. To achieve our goal we need an elementary geometric lemma. To state and prove it we introduce some notation. For a point $(a,b)\in \mathbb{R}^2$ put $T_{(a,b)} = (a,b)+ U$ and $\bar{T}_{(a,b)}=(a,b)+ \bar{U}$, where $+$ means here translation. The squares $T_{(a,b)}, \; (a,b)\in \Z^2$ are disjoint and their union cover $\mathbb{R}^2$. As $A_{\varphi}$ is a rotation the sets $A_{\varphi}(T_{(a,b)})$ are squares too with the same properties. Thus there exists for any $(n,m)\in \mathbb{Z}^2$ unique $(a,b)\in \Z^2$ such that $(n,m) \in A_{\varphi}(T_{(a,b)})$.


\begin{lemma} \label{l:surj}
	Let $(n,m)\in \Z^2.$ Then $(n,m)\neq r_{\varphi}(a,b)$ for each $(a,b)\in \Z^2$ if and only if $(n,m)\in
A_{\varphi}(T_{(a,b)})$ and $A_{\varphi}(a+\varepsilon_a, b+\varepsilon_b)\notin T_{(n,m)}$ for each
$\varepsilon_a,\varepsilon_b\in \{0,1\}.$ Moreover in this case the points $A_{\varphi}(a+\varepsilon_a, b+\varepsilon_b), \varepsilon_a,\varepsilon_b\in \{0,1\}$ belong in some order to the horizontal and vertical neighbour squres to $T_{(n,m)}$.
\end{lemma}

\begin{proof}
Necessity: If $A_{\varphi}(a+\varepsilon_a, b+\varepsilon_b)\in T_{(n,m)}$ for some
$\varepsilon_a,\varepsilon_b\in \{0,1\},$ then $r_{\varphi}(a+\varepsilon_a, b+\varepsilon_b)=(n,m),$ thus $(n,m)$
is the image of some point.

\bigskip

Sufficiency: We have $T_{(a,b)}=(a,b)+T_{(0,0)},$ therefore
$A_{\varphi}(T_{(a,b)})=A_{\varphi}(a,b)+A_{\varphi}(T_{(0,0)}),$ since $A_{\varphi}$ is linear. The same holds for
the closure of $T_{(a,b)},$ denoted by $\overline{T_{(a,b)}}.$ In Figure \ref{fig:1} we show the four main situations of the rotated unitgrid.

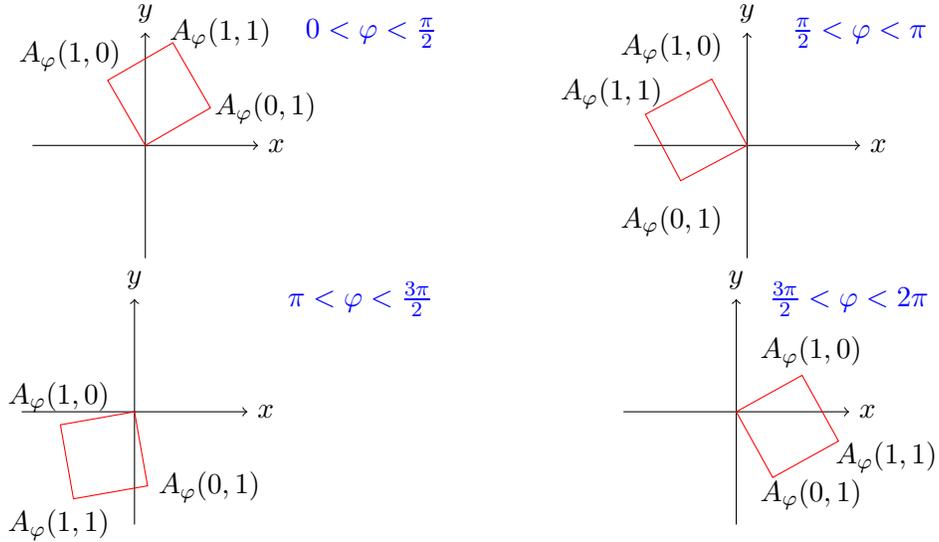
\begin{figure}[htpb]
\begin{tikzpicture}[scale=1.0]
	\draw[->] (-1.5,0)--(1.5,0) node[right]{$x$};
	\draw[->] (0,-1.5)--(0,1.5) node[above]{$y$};
	\draw (1.6,0.5) node{$A_{\varphi}(0,1)$};
	\draw (1,1.5) node{$A_{\varphi}(1,1)$};
	\draw (-1,1.2) node{$A_{\varphi}(1,0)$};
	\draw[blue] (3,1.5) node{$0<\varphi<\frac{\pi}{2}$};
    \draw[red,rotate around={30:(0,0)}] (0,0) rectangle (1,1);

	\draw[->] (6.5,0)--(9.5,0) node[right]{$x$};
	\draw[->] (8,-1.5)--(8,1.5) node[above]{$y$};
	\draw (7,-1) node{$A_{\varphi}(0,1)$};
	\draw (6.2,0.7) node{$A_{\varphi}(1,1)$};
	\draw (7,1.3) node{$A_{\varphi}(1,0)$};
	\draw[blue] (9.5,1.5) node{$\frac{\pi}{2}<\varphi<\pi$};
	\draw[red,rotate around={118:(8,0)}] (8,0) rectangle (9,1);
\end{tikzpicture}

\begin{tikzpicture}[scale=1.0]
	\draw[->] (-1.5,0)--(1.5,0) node[right]{$x$};
	\draw[->] (0,-1.5)--(0,1.5) node[above]{$y$};
	\draw[red,rotate around={190:(0,0)}] (0,0) rectangle (1,1);
	\draw (1,-1) node{$A_{\varphi}(0,1)$};
	\draw (-1,-1.5) node{$A_{\varphi}(1,1)$};
	\draw (-1,0.2) node{$A_{\varphi}(1,0)$};
	\draw[blue] (3,1.5) node{$\pi<\varphi<\frac{3\pi}{2}$};
\draw[->] (6.5,0)--(9.5,0) node[right]{$x$};
\draw[->] (8,-1.5)--(8,1.5) node[above]{$y$};
\draw[red,rotate around={299:(8,0)}] (8,0) rectangle (9,1);
\draw (9,-1.1) node{$A_{\varphi}(0,1)$};
\draw (10,-0.6) node{$A_{\varphi}(1,1)$};
\draw (9,0.8) node{$A_{\varphi}(1,0)$};
	\draw[blue] (9.5,1.5) node{$\frac{3\pi}{2}<\varphi<2\pi$};
\end{tikzpicture}
\caption{The situation of $A_{\varphi}(T_{(0,0)})$}
\label{fig:1}
\end{figure}

We give the proof of the lemma in detail only for the case $\frac{3\pi}{2}<\varphi<2\pi$, the other cases can be handled similarly. Assume that $(n,m)\in A_{\varphi}(T_{(a,b)})$, but $A_{\varphi}(a+\varepsilon_a, b+\varepsilon_b)\notin T_{(n,m)}$ for each $\varepsilon_a,\varepsilon_b\in \{0,1\}.$ Then we have
$$A_{\varphi}(a,b+1)_y < m, \mbox{   } A_{\varphi}(a+1,b)_y\geq m+1,$$
$$A_{\varphi}(a,b)_x < n, \mbox{   } A_{\varphi}(a+1,b+1)_x\geq n+1.$$
Here and in the sequel, $A_{\varphi}(.,.)_x, A_{\varphi}(.,.)_y$ denote the $x$ (respectively $y$) coordinate of the corresponding point. Notice that the two strong inequalities are due to the assumption $A_{\varphi}(a+\varepsilon_a, b+\varepsilon_b)\notin T_{(n,m)}$ for each $\varepsilon_a,\varepsilon_b\in \{0,1\}.$

We show that $T_{(n,m)}\setminus A_{\varphi}(\overline{T_{(a,b)}})$ is the union of three disjoint triangles
$H_1,H_2,H_3$. The triangle $H_1$ is bordered by the lines $x=n, y=m+1$ and by the line sequent between the points  $A_{\varphi}(a,b)$ and $A_{\varphi}(a+1,b)$. Similarly $H_2$ is bordered by the lines $x=n+1, y=m+1$ and by the line sequent between the points $A_{\varphi}(a+1,b)$ and $A_{\varphi}(a+1,b+1)$. Finally the borders of $H_3$ are the lines $x=n,y=m+1$ and the line segment between the points $A_{\varphi}(a,b+1)$ and $A_{\varphi}(a+1,b+1)$.

That are proper triangles. For example, look at $H_1.$ The triangle with vertices $A_{\varphi}(a+1,b), \; A_{\varphi}(a+1,b+1)$ and the intersection of the lines $x=A_{\varphi}(a+1,b)_x$ and $y = A_{\varphi}(a+1,b+1)_y$ is rectangular and the legth of its hypotenuse is $1$, thus $A_{\varphi}(a+1,b+1)_x- A_{\varphi}(a+1,b)_x <1$, which together with the above inequalities implies $A_{\varphi}(a+1,b)_x> n$. As $A_{\varphi}(a,b)_x < n$ the line $x=n$ has an
intersection with the line segment between $A_{\varphi}(a,b)$ and $A_{\varphi}(a+1,b),$ which is different from the end points. Similarly we have $A_{\varphi}(a,b)_y - A_{\varphi}(a+1,b)_y<1,$ i.e.~$A_{\varphi}(a,b)_y<m+1,$ hence the line $y=m+1$ intersects the line segment between $A_{\varphi}(a,b)$ and $A_{\varphi}(a+1,b).$ In this case it may happen that the intersection point is $A_{\varphi}(a+1,b).$

As a byproduct we proved $A_{\varphi}(a+1,b)\in T_{(n,m+1)}$, i.e $r_{\varphi}(a+1,b) = (n,m+1)$. Performing similar arguments for $H_2$ and $H_3$ we obtain that $A_{\varphi}(a+1,b+1)\in T_{(n+1,m)}$, i.e $r_{\varphi}(a+1,b+1) = (n+1,m)$ and $A_{\varphi}(a,b+1)\in T_{(n,m-1)}$, i.e $r_{\varphi}(a,b+1) = (n,m-1)$ respectively. Finally, one can prove $A_{\varphi}(a,b)\in T_{(n-1,m)}$, i.e $r_{\varphi}(a,b) = (n-1,m)$ too, hence the second assertion is proved.

\medskip

No we finalize the proof of the first assertion. We assume that there exists $(c,d)\in \Z^2,$ such that $r_{\varphi}(c,d)=(n,m).$ Then $A_{\varphi}(c,d)\in
T_{(n,m)}$ and $(c,d)\notin \overline{T_{(a,b)}},$ i.e.~$A_{\varphi}(c,d)\in T_{(n,m)}\setminus
A_{\varphi}(\overline{T_{(a,b)}}),$ hence $A_{\varphi}(c,d)$ is contained in one of the triangles $H_1,H_2,H_3.$
We have $H_1\subseteq A_{\varphi}(T_{(a,b-1)}),$ which contains only the grid point
$A_{\varphi}(a,b-1),$ therefore $(c,d)=(a,b-1).$ In contrast, $A_{\varphi}(a,b-1)_x< A_{\varphi}(a,b)< n,$
i.e.~$\lfloor A_{\varphi}(a,b-1)_x\rfloor < n= \lfloor A_{\varphi}(c,d)_x\rfloor,$ which is a contradiction. We
have $H_2\subseteq A_{\varphi}(T_{(a+1,b)})$ and $H_3 \subseteq A_{\varphi}(T_{(a,b+1)}),$ which implies
$(c,d)=(a+1,b)$ and $(c,d)=(a,b+1)$ respectively. The proof of the case $\frac{3\pi}{2}<\varphi<2\pi$ is finished. As we mentioned above the three other cases can be handled similarly.
\end{proof}

Similarly to Section \ref{s:inj} we introduce the function $N_{\varphi}(M)$, which is the number of grid points $(n,m)$, such that $|n|,|m|\le M $ and which are images of no grid points under the mapping $r_{\varphi}$. If $\varphi = k\frac{\pi}{2}, k \in \Z$ then $r_{\varphi}$ is bijective, hence $N_{\varphi}(M)=0$. By the next theorem this cannot happen otherwise.

\begin{thm} \label{t:surj}
If $\varphi \not= k\frac{\pi}{2}, k \in \Z$ then there exist constants $c_5,c_6>0$ depending only on $\varphi$ such that
$c_5M^2 \le T_{\varphi}(M) \le c_6M^2$ except when $\cos\varphi = \pm \sin\varphi + r,\; r\in \Q$, in which case $c_5M \le T_{\varphi}(M) \le c_6M$ hold.
\end{thm}

An immediate consequence is
\begin{cor}
  If $\varphi\not= k\frac{\pi}{2}$ then $r_{\varphi}$ has infinitely many different orbits.
\end{cor}

\noindent
{\it Proof of Theorem \ref{t:surj}.}

  Let $(n,m)\in \Z^2$, which is the image of no grid points under $r_{\varphi}$. Then, by Lemma \ref{l:surj}, there exists $(a,b)\in \Z^2$ such that $(n,m)\in A_{\varphi}(T_{(a,b)})$ and $A_{\varphi}(a+\varepsilon_a,b+\varepsilon_b)\notin T_{(n,m)}$ for all   $\varepsilon_a,\varepsilon_b \in \{0,1\}$. In the same lemma we proved that $A_{\varphi}(a,b),A_{\varphi}(a,b+1),A_{\varphi}(a+1,b),A_{\varphi}(a+1,b+1)$ belong in some order to the four unit squares left, right, top and down to the unit square $T_{(n,m)}$. Depending on the size of $\varphi$ we distinguish four cases.

\medskip

  {\bf Case 1. $0\le \varphi< \frac{\pi}{2}$. }
  Then by Lemma \ref{l:surj} $A_{\varphi}(a,b)\in T_{(n,m-1)}, A_{\varphi}(a,b+1)\in T_{(n-1,m)}, A_{\varphi}(a+1,b)\in T_{(n+1,m)},
  A_{\varphi}(a+1,b+1)\in T_{(n,m+1)}$, which means
  \begin{eqnarray*}
    0 &\le& a\cos \varphi - b \sin \varphi -n < 1 \\
    0 &\le& a\cos \varphi - (b+1) \sin \varphi -(n-1) < 1 \\
    0 &\le& (a+1)\cos \varphi - b \sin \varphi -(n+1) < 1 \\
    0 &\le& (a+1)\cos \varphi - (b+1) \sin \varphi -n < 1.
  \end{eqnarray*}
  Rearranging we obtain the system of inequalities
  \begin{eqnarray}
    0 &\le& a\cos \varphi - b \sin \varphi -n < 1 \label{e:egy}\\
    \sin \varphi -1 &\le& a\cos \varphi - b \sin \varphi -n < \sin \varphi \label{e:kettő}\\
    1-\cos \varphi &\le& a\cos \varphi - b \sin \varphi -n < 2 -\cos \varphi \label{e:három}\\
    \sin \varphi-\cos \varphi &\le& a\cos \varphi - b \sin \varphi -n < 1+\sin \varphi-\cos \varphi.
    \label{e:négy}
  \end{eqnarray}

  As $0\le \varphi< \frac{\pi}{2}$ we have $\sin \varphi, \cos \varphi>0$. Under this assumption we have $\sin \varphi -1<
  \sin \varphi-\cos \varphi < 1-\cos \varphi$ and $0< 1-\cos \varphi$ holds too. Hence $\max \{\sin \varphi -1,
  \sin \varphi-\cos \varphi , 1-\cos \varphi,0\} = 1-\cos \varphi$.

\medskip

  Similarly, $\sin \varphi<1 < 2 -\cos \varphi$ and $\sin \varphi< 1+\sin \varphi-\cos \varphi$, thus
  $\min\{\sin \varphi,1 , 2 -\cos \varphi, 1+\sin \varphi-\cos \varphi\} = \sin \varphi$. Hence the inequalities
  \eqref{e:egy}-\eqref{e:négy} hold if and only if
$$
     1-\cos \varphi \le a\cos \varphi - b \sin \varphi -n < \sin \varphi.
$$
After multiplying by $2$ and adding $\cos \varphi-\sin \varphi -1$ we obtain
  \begin{equation}\label{five}
   1- \cos \varphi -\sin \varphi \le (2a+1)\cos \varphi - (2b+1) \sin \varphi -2n-3 <  \sin \varphi + \cos \varphi-1.
  \end{equation}

  Performing the analogous computation for $m$ we get that $a,b,m\in \Z$ satisfy the requirements if and only if
  \begin{equation}\label{six}
1- \cos \varphi -\sin \varphi \le (2a+1)\sin\varphi + (2b+1) \cos \varphi -2m-1 <  \sin \varphi + \cos \varphi-1.  \end{equation}
  As $\sin \varphi + \cos \varphi-1>0$, hence $1-\sin \varphi - \cos \varphi <0$  we can apply Lemma \ref{l:main} tothe system of inequalities \eqref{five} and \eqref{five} with $x_1=2a+1$ and $x_2=2b+1$, which proves the theorem in this case.

\bigskip

 {\bf Case 2. $\frac{\pi}{2} \le \varphi< \pi$. }
  Then by Lemma \ref{l:surj} $A_{\varphi}(a,b)\in T_{(n+1,m)}, A_{\varphi}(a,b+1)\in T_{(n,m-1)},
 A_{\varphi}(a+1,b)\in T_{(n,m+1)}, A_{\varphi}(a+1,b+1)\in T_{(n-1,m)}$. Then
  \begin{eqnarray*}
   	0 &\le& a\cos \varphi - b \sin \varphi -(n+1) < 1 \\
  	0 &\le& a\cos \varphi - (b+1) \sin \varphi -n < 1 \\
    0 &\le& (a+1)\cos \varphi - b \sin \varphi -n < 1 \\
    0 &\le& (a+1)\cos \varphi - (b+1) \sin \varphi -(n-1) < 1.
  \end{eqnarray*}
  Rearranging the inequalities we obtain
  \begin{eqnarray}
  	\label{9}
  1 &\le& a\cos \varphi - b \sin \varphi -n < 2 \\
  \label{10}
  \sin \varphi  &\le& a\cos \varphi - b \sin \varphi -n < 1+\sin \varphi \\
  	\label{11}
  -\cos \varphi &\le& a\cos \varphi - b \sin \varphi -n < 1 -\cos \varphi\\
  	\label{12}
  \sin \varphi-\cos \varphi -1 &\le& a\cos \varphi - b \sin \varphi -n < \sin \varphi-\cos \varphi.
  \end{eqnarray}

\medskip
  We have $\sin \varphi>0, \cos \varphi<0$, because $\frac{\pi}{2}< \varphi<\pi $.

  Thus $\max \{\sin \varphi, 1, -\cos \varphi , \sin \varphi-\cos \varphi -1\} = 1$. Similarly, $\min\{2, 1+\sin \varphi,1 -\cos \varphi, \sin \varphi -\cos \varphi\} = \sin\varphi- \cos \varphi$.
  Hence the inequalities \eqref{9}-\eqref{12} hold if and only if
  $$
    1 \le a\cos \varphi - b \sin \varphi -n < \sin\varphi- \cos \varphi.
$$
Multiplying by $2$ and adding $-\sin \varphi+\cos \varphi -1$ we obtain
  \begin{equation*}
     1+ \cos \varphi -\sin \varphi \le (2a+1)\cos \varphi - (2b+1) \sin \varphi -2n-1 <  \sin \varphi - \cos \varphi-1.
  \end{equation*}

   Performing the analogous computation for $m$ we get that $a,b,m\in \Z$ satisfy the requirements if and only if
  \begin{equation*}
     1+ \cos \varphi -\sin \varphi \le (2a+1)\sin \varphi - (2b+1) \cos \varphi -2m-1 <  \sin \varphi - \cos \varphi-1.
  \end{equation*}

As $\sin \varphi>0$ and $\cos \varphi<0$ we have $\sin \varphi - \cos \varphi-1>0$, hence $1- \sin \varphi + \cos \varphi<0$, thus we may apply Lemma \ref{l:main} to the last system of inequalities, which completes the proof in the second case,

\bigskip

  {\bf Case 3. $\pi < \varphi< \frac{3\pi}{2}$.}
  Then by Lemma \ref{l:surj}  $A_{\varphi}(a,b)\in T_{(n,m+1)}, A_{\varphi}(a,b+1)\in T_{(n+1,m)},
  A_{\varphi}(a+1,b)\in T_{(n-1,m)}, A_{\varphi}(a+1,b+1)\in T_{(n,m-1)}$. The same computation as in Cases 1. and 2. lead to the system of inequalities

  \begin{eqnarray*}
   	1+ \sin \varphi + \cos \varphi &\le& (2a+1)\cos \varphi - (2b+1) \sin \varphi -2n-1 < -\sin \varphi - \cos \varphi -1\\
  	1+ \sin \varphi + \cos \varphi &\le& (2a+1)\cos \varphi + (2b+1) \sin \varphi -2m-1 < -\sin \varphi - \cos \varphi -1
  \end{eqnarray*}

  As $\sin \varphi, \cos \varphi,1+ \sin \varphi + \cos \varphi<0$ we have $-\sin \varphi - \cos \varphi -1>0$ and can apply Lemma \ref{l:main} again.

 \bigskip

 {\bf Case 4. $\frac{3\pi}{2} < \varphi< 2\pi$.}
  Then by Lemma \ref{l:surj}  $A_{\varphi}(a,b)\in T_{(n-1,m)}, A_{\varphi}(a,b+1)\in T_{(n,m+1)},
  A_{\varphi}(a+1,b)\in T_{(n,m-1)}, A_{\varphi}(a+1,b+1)\in T_{(n+1,m)}$. The same computation as in Cases 1. and 2. lead to the system of inequalities

  \begin{eqnarray*}
   	1+ \sin \varphi - \cos \varphi &\le& (2a+1)\cos \varphi - (2b+1) \sin \varphi -2n-1 < -\sin \varphi + \cos \varphi -1\\
  	1+ \sin \varphi - \cos \varphi &\le& (2a+1)\cos \varphi + (2b+1) \sin \varphi -2m-1 < \sin \varphi - \cos \varphi -1
  \end{eqnarray*}

  As $\sin \varphi, -\cos \varphi,1+ \sin \varphi - \cos \varphi<0$ we have $-\sin \varphi + \cos \varphi -1>0$ and can apply Lemma \ref{l:main} again, which completes the proof of the theorem.
 $\Box$

\section{Orbits with short periodicity}

Our first goal was to study the periodicity of the orbits of $r_{\varphi}$, which seems to be very difficult. As a first step we examined some other properties of $r_{\varphi}$, but did not forget the ultimate goal. In this section we present a small finding, which corresponds to $\varphi=\frac{\pi}{4}.$ We show that there are infinitely many $a\in \Z,$ such that the orbit of $r_{\varphi}$ generated by $(a,0)$ is periodic of length $8,$ i.e. $$r_{\varphi}^8(a,0)=(a,0).$$

Nevertheless we could present other examples, but this already shows that we do not have yet the necessary technique to prove much more general results.

In the next lemma we collected those identities, which are necessary to prove our periodicity result. Their proofs are one step direct computation. We denote, as usual in these notes, the fractional part of $x$ by $\{x\}.$

\begin{lemma}
	Let $a\in \Z$ and set $\omega=\lfloor \frac{1}{\sqrt{2}}a\rfloor.$ Suppose $\lfloor \sqrt{2}\omega\rfloor =a-1.$ Then

	 \begin{equation}\label{1}
	 \left\lfloor -\frac{1}{\sqrt{2}}a+\frac{1}{\sqrt{2}}\right\rfloor=-\omega	
	\end{equation}
	
	\begin{equation}\label{2}
	\left\lfloor \frac{1}{\sqrt{2}}a-\frac{1}{\sqrt{2}}\right\rfloor=\begin{cases}
	\omega, \mbox{ if } a=1	\\
	\omega-1, \mbox{ otherwise  }
	\end{cases}\end{equation}
	
	\begin{equation}\label{4}
	\left\lfloor \frac{1}{\sqrt{2}}a+\frac{1}{\sqrt{2}}\right\rfloor=\begin{cases}
	\omega, \mbox{ if } \{\frac{1}{\sqrt{2}}a\}< 1-\frac{1}{\sqrt{2}}	\\
	\omega+1, \mbox{ if } \{\frac{1}{\sqrt{2}}a\}\geq 1-\frac{1}{\sqrt{2}}	
	\end{cases}\end{equation}
	
		\begin{equation}\label{5}
	\left\lfloor -\frac{1}{\sqrt{2}}a+\sqrt{2}\right\rfloor=\begin{cases}
	-\omega, \mbox{ if } \{\frac{1}{\sqrt{2}}a\}> \sqrt{2}-1	\\
	-\omega+1, \mbox{ if } \{\frac{1}{\sqrt{2}}a\}\leq \sqrt{2}-1	
	\end{cases}\end{equation}

		\begin{equation}\label{7}
	\left\lfloor \sqrt{2}\omega+\sqrt{2}\right\rfloor=\begin{cases}
	a, \mbox{ if } \{\sqrt{2}\omega\} +\{\sqrt{2}\}<1	\\
	a+1, \mbox{ if } \{\sqrt{2}\omega\} +\{\sqrt{2}\}\geq 1
	\end{cases}\end{equation}
	
		\begin{equation}\label{8}
	\left\lfloor -\sqrt{2}\omega+\frac{1}{\sqrt{2}}\right\rfloor=\begin{cases}
	-a, \mbox{ if } \{\sqrt{2}\omega\} > \frac{1}{\sqrt{2}}\\
	-a+1, \mbox{ if } \{\sqrt{2}\omega\} \leq \frac{1}{\sqrt{2}}
	\end{cases}\end{equation}
	
\end{lemma}

\begin{thm}\label{45fok}
	Let $a\in \Z,$ $\omega=\lfloor \frac{1}{\sqrt{2}}a\rfloor$ and suppose $\lfloor \sqrt{2}\omega\rfloor = a-1.$
If $\{\frac{1}{\sqrt{2}}a\}\in \left[1-\frac{1}{\sqrt{2}}, \frac{1}{\sqrt{2}}\right],$ then $r_{\varphi}^8 (a,0)=(a,0).$ There exist infinitely many $a\in \Z$ satisfying the assumptions.
\end{thm}

\begin{proof}
First we prove that the assumptions imply $\{\sqrt{2}\omega\}\in
\left[1-\frac{1}{\sqrt{2}}, \frac{1}{\sqrt{2}}\right]$. Indeed, $[\sqrt{2}\omega] = a-1$ means $a-1 \le \sqrt{2}\omega <a$ (equality is only possible if $a=1$), hence $0< \frac{a}{\sqrt{2}} -\omega <\frac{1}{\sqrt{2}}$. Further  $\{\frac{1}{\sqrt{2}}a\}\in \left[1-\frac{1}{\sqrt{2}}, \frac{1}{\sqrt{2}}\right]$ is equivalent to the sequence of inequalities
\begin{eqnarray*}
1- \frac{1}{\sqrt{2}} &<& \frac{a}{\sqrt{2}} - \omega < \frac{1}{\sqrt{2}}\\
-\frac{1}{\sqrt{2}} &<&\sqrt{2}\omega - a < -1 + \frac{1}{\sqrt{2}}\\
1- \frac{1}{\sqrt{2}} &<& \sqrt{2}\omega -(a-1)< \frac{1}{\sqrt{2}},
\end{eqnarray*}
which proves the claim.

Now we prove the second assertion. As $\sqrt{2}$ is irrational, the sequence $\left\{\frac{a}{\sqrt{2}}\right\}$ is by Theorem \ref{th:Kronecker} uniformly distributed, thus there are infinitely many $a\in \Z$ satisfying $\{\frac{1}{\sqrt{2}}a\}\in \left[1-\frac{1}{\sqrt{2}}, \frac{1}{\sqrt{2}}\right]$.

Now we turn to prove the first assertion. It is obvious that $r_{\varphi}(a,0)=(\omega,\omega).$ Further by the assumption $\lfloor \sqrt{2}\omega\rfloor
= a-1$ we have $r_{\varphi}^2(a,0)=(0,a-1).$ Further we have
$r_{\varphi}^3(a,0)=r_{\varphi}(0,a-1)=(\lfloor-\frac{1}{\sqrt{2}}a+\frac{1}{\sqrt{2}}\rfloor, \lfloor
\frac{1}{\sqrt{2}}a-\frac{1}{\sqrt{2}}\rfloor).$ Thus by equations \eqref{1} and \eqref{2} we get
$r_{\varphi}^3=(-\omega,\omega-1).$ Further, by equation \ref{8} we have $r_{\varphi}^4(a,0)=(-a+1,-1).$ After that
we get $r_{\varphi}^5(a,0)=(-\omega+1,-\omega-1)$ by equation \eqref{5}. Then $r_{\varphi}^6(a,0)=(1,-a).$ Further we
get by equations \eqref{1} and \eqref{4} $r_{\varphi}^7(a,0)=(\omega+1,-\omega-1).$ Finally, using equation \eqref{7} we
have $r_{\varphi}^8(a,0)=r_{\varphi}(\omega+1,-\omega-1)=(\lfloor \sqrt{2}\omega + \sqrt{2}\rfloor,0)=(a,0).$
\end{proof}

\begin{rem}
	There exist infinitely many natural numbers $a,$ which fulfill the conditions in Theorem \ref{45fok}, therefore
there are infinitely many orbits with short periodicity.
\end{rem}

	We only give one class of starting points for the rotation by $45^{\circ}$ in order to achieve periodicity of
$8.$ There may exist many other starting points with the same periodicity.

\section{Remarks on $ROUND\circ A_{\varphi}$} \label{s:round}

Following the proofs of Theorems \ref{t:inj} and \ref{t:surj} one can prove similar statements for the function $ROUND\circ A_{\varphi}$. To perform such a project one has to adjust Lemma \ref{l:main} according the new rounding function. This is straight forward if one of $\cos \varphi$ and  $\sin \varphi$ is irrational, but needs some computation otherwise. To save space we leave this to the interested reader.

\end{document}